\documentclass[a4paper,10pt]{article}
\usepackage{latexsym,amssymb,amsfonts,amsmath,amsthm,mathrsfs,amstext,color,graphicx,times}
\usepackage[mathscr]{euscript}
\usepackage{indentfirst}
\usepackage{cite}
\usepackage{wrapfig}
\usepackage[all]{xy}
\usepackage{tikz}
\usepackage{hyperref}
\usepackage{color}
\usetikzlibrary{patterns}

\newcommand{\R}{\mathbb{R}}
\newcommand{\N}{\mathbb{N}}

\newtheorem{theorem}{Theorem}[section]
\newtheorem{corollary}[theorem]{Corollary}
\newtheorem{proposition}{Proposition}[section]
\newtheorem{definition}{Definition}[section]
\newtheorem{remark}{Remark}
\newtheorem{lemma}{Lemma}[section]
\newtheorem{example}{Example}[section]
\newcommand{\tos}{\rightrightarrows} % point-to-set mappings
\DeclareMathOperator{\co}{co}
\DeclareMathOperator{\dom}{dom}

\DeclareMathOperator{\Int}{int}
\DeclareMathOperator{\argmin}{argmin}
\DeclareMathOperator{\qep}{QEP}
\DeclareMathOperator{\ep}{EP}
\DeclareMathOperator{\vip}{VIP}
\DeclareMathOperator{\mqep}{MQEP}
\DeclareMathOperator{\mep}{MEP}

\DeclareMathOperator{\fix}{Fix}

\title{Finite intersection property for bifunctions and existence for quasi-equilibrium problems}
\author{John Cotrina
	\thanks{Universidad del Pac\'ifico, Lima, Per\'u. Email: 
	\texttt{ cotrina\_je@up.edu.pe}}
	\and Anton Svensson 
	\thanks{Mathematics Department, Universidad de Chile, Santiago, Chile and Lab. PROMES UPR CNRS 8521, University of Perpignan, France. Email: \texttt{asvensson@dim.uchile.cl}}
}	

\begin{document}
\maketitle
\begin{abstract}
The ``finite intersection property" for bifunctions is introduced and  its relationship with  generalized monotonicity properties is studied. Some results concerning existence of solution for  (quasi-)equilibrium problems  are established and several results well-known in the literature are recovered. Furthermore, two applications are considered.

\bigskip

\noindent{\bf Keywords:} Quasi-equilibrium problem, Set-valued map, Generalized monotonicity, Finite intersection property
\bigskip

\noindent{{\bf MSC (2010)}: 47J20, 49J35, 90C37} 
\end{abstract}

\section{Introduction}

Given a non-empty subset $C$ of a topological space $X$ and a bifunction $f:X\times X\to\R$, the \emph{equilibrium problem}, see Blum and Oettli \cite{OB93}, is the problem of finding:
\begin{align}\label{EP}
x\in C\mbox{ such that }f(x,y) \geq0\mbox{ for all }y\in C.\tag{EP}
\end{align}

Problem \eqref{EP} has been extensively studied in recent years and several existence results have been developed under generalized monotonicity (see \cite{BP01,BP05,BS96,OB93,castellani2012,Chaldi2000,JCYG,FFB2001,Iusem2009,IUSEM2003621,PQK-NHQ-19} and the  references therein). A recurrent theme in the analysis of an equilibrium problem is its relation with the so-called \emph{Minty equilibrium problem}, which corresponds to a sort of dual formulation of the equilibrium problem and consists of finding:
\begin{align}\label{MEP}
x\in C\mbox{ such that }f(y,x) \leq0\mbox{ for all }y\in C.\tag{MEP}
\end{align}

It is well-known that every solution of \eqref{EP} is a solution of \eqref{MEP}, provided that the bifunction $f$ is pseudo-monotone. Recently, it was showed in \cite{JC-JZ-2018} that
the converse holds  under pseudo-monotonicity of  $-f$.

A generalization of \eqref{EP} is the so-called \emph{quasi-equilibrium problem}, in which the constraint set depends on the currently analysed point.
More precisely, given a bifunction $f:X\times X\to\R$ and a  set-valued map $K:C\tos C$, the quasi-equilibrium problem consists of finding:
\begin{align}\label{QEP}
x\in K(x)\mbox{ such that }f(x,y)\geq0 \mbox{ for all }y\in K(x).\tag{QEP}
\end{align}
The associated \emph{Minty quasi-equilibrium problem} consists of finding:
\begin{align}\label{MQEP}
x\in K(x)\mbox{ such that }f(y,x)\leq0 \mbox{ for all }y\in K(x).\tag{MQEP}
\end{align}

In \cite{ACI,JC-JZ-2018} some existence results for quasi-equilibrium problems were shown under generalized monotonicity. 

Our aim in this work is to establish sufficient conditions for the existence of solutions for the quasi-equilibrium problem. These conditions are based on some new concepts that we introduce related to bifunctions, the so-called finite intersection properties, and which are inspired from \cite{Nessah-Tian}. Additionally, we discuss the relation of these conditions with generalized monotonicity properties.

This paper is organized as follows. In Section \ref{preliminaries} we recall some definitions about generalized convexity, generalized monotonicity for bifunctions and continuity for set-valued maps, also we state the well-known Kakutani's fixed point theorem. 
In Section \ref{FIP} the concept of finite intersection property for bifunctions is introduced and its relation with generalized monotonicity properties are studied. 
In Section \ref{main1}, it is shown that the Minty equilibrium problem derived from a bifunction is solvable provided that this bifunction has the finite intersection property.
Our main result  is established in Section \ref{main2}.
Finally, in Section \ref{applications} we consider as applications the variational inequality problem and the generalized Nash equilibrium problem.

\section{Preliminaries and notations}\label{preliminaries}

Assume $X$ is a (real) topological vector space. Let $S$ be a subset of  $X$. The convex hull and the interior  of $S$ will be denoted by $\co(S)$ and $\Int(S)$, respectively.  

Let us recall some classical definitions of generalized convexity. A function $h:X\to\R\cup\{+\infty\}$ is said to be \emph{quasi-convex} if, for any $\lambda\in\R$, its sublevel set $S_\lambda[h]$ is convex, where
\[
S_\lambda[h]:=\{x\in X:~h(x)\leq \lambda\},
\]
and $h$ is said to be \emph{semi-strictly quasi-convex} if, it is quasi-convex and, for any $x,y\in X$ such that $h(x)\neq h(y)$ 
the following holds
\[
h(tx+(1-t)y)< \max\{h(x),h(y)\}\mbox{ for all }t\in]0,1[.
\]
Roughly speaking, a semi-strictly quasi-convex function is a quasi-convex function which does not admit a flat part, except possibly $\argmin_X f$.

Now we recall some notions of generalized monotonicity for bifunctions.
Given a subset $C$ of $X$, a bifunction $f:X\times X\to\R$ is said to be:
\begin{itemize}
\item \emph{cyclically monotone} on  $C$ if, for all $n\in\N$ and all $x_0,x_1,\dots,x_n\in C$ the following holds
\[
\sum_{i=1}^n f(x_i,x_{i+1})\leq0,
\]
with $x_{n+1}=x_0$;
\item \emph{cyclically quasi-monotone} on $C$ if, for all $n\in\N$ and all $x_0,x_1,\dots,x_n\in C$, there exists $i\in\{0,1,\dots,n\}$ such that
\[
f(x_i,x_{i+1})\leq0,
\]
where $x_{n+1}=x_0$;

\item \emph{monotone} on $C$ if, for all $x,y\in C$, 
\[
f(x,y)+f(y,x)\leq0;
\]
\item \emph{pseudo-monotone} on $C$ if, for all $x,y\in C$, %the following implication holds
\[
f(x,y)\geq 0\Rightarrow f(y,x)\leq0;
\]
\item \emph{quasi-monotone} on $C$  if, for all $x,y\in C$, %the following implication holds
\[
f(x,y)>0\Rightarrow f(y,x)\leq0.
\]
\end{itemize}

The following diagram shows the relationship between the generalized monotonicity properties that we have considered.

\vspace{.2cm}

\begin{center}
\begin{tabular}{ccccc}
cyclic monotonicity &$\Rightarrow$& monotonicity&$\Rightarrow$&pseudo-monotonicity\\
$\Downarrow$&&&&$\Downarrow$\\
cyclic quasi-monotonicity &&$\Rightarrow$& &quasi-monotonicity
\end{tabular}
\end{center}

\vspace{.2cm}

We give now a characterization cyclic quasi-monotonicity that will be useful in the next section.

\begin{proposition}\label{cyclic-q-fip}
	Let $C$ be a non-empty subset of $X$ and $f:X\times X\to\R$ be a bifunction. Then, $f$ is cyclic quasi-monotone on $C$ if and only if for any finite and non-empty subset $A$ of $C$, there exists $x\in A$ such that
	\[
	\max_{a\in A}f(a,x)\leq 0.
	\]
\end{proposition}
\begin{proof}
	Let us first prove the direct implication. Reasoning by contradiction, suppose that there exists $A=\{x_1,x_2,\dots,x_n\}\subset C$ such that $(\bigcap_{i=1}^n F_{x_i})\cap A=\emptyset$, where $F_{x_i}=\{y\in C:~f(x_i,y)\leq0\}$. This is equivalent to
	\begin{align}\label{cqm-fip}
	\left(\bigcup_{i=1}^n F_{x_i}^c\right)\cup A^c=C.
	\end{align}
	Set $x_{i(1)}=x_1$, equality \eqref{cqm-fip} implies that there exists $x_j$ with $x_j\neq x_1$ such that $x_1\in F_j^c$, that means $f(x_j,x_1)>0$. We set $x_{i(2)}=x_j$ and apply the equality \eqref{cqm-fip} again. Continuing in this way, we define a sequence $(x_{i(n)})_{n\in\N}$ such that
	\begin{align}\label{cqp-fip2}
	f(x_{i(k+1)},x_{i(k)})>0
	\end{align}
	for all $k\in\N$.
	
	Since the set $\{x_1,x_2,\dots,x_n\}$ is finite, there exist $m,k\in\N$ with $m<k$ such that $x_{i(k+1)}=x_{i(m)}$. We now consider the points
	\[
	\hat{x}_1=x_{i(m)},~\hat{x}_2=x_{i(k)},~\hat{x}_3=x_{i(k-1)},\dots, \hat{x}_{k+1-m}=x_{i(m+1)}
	\]
	which, due to the inequality \eqref{cqp-fip2}, satisfy
	\[
	f(\hat{x}_j,\hat{x}_{j+1})>0
	\]
	for all $j=1,\dots,k+1-m$, with $\hat{x}_{k+2-m}=\hat{x}_1$. This means that $f$ is not cyclic quasi-monotone and we get a contradiction. 

	Now we prove the converse implication. Given points $x_1,...,x_{n+1}\in C$ with $x_{n+1}=x_1$, consider the finite and non-empty set $A:=\{x_i: i=1,...,n\}$. We know there exists $x\in A$ such that $f(a,x)\leq 0$ for all $a\in A$. Of course $x=x_{i}$ for some $i=1,...,n$ and in particular we have $f(x_i,x_{i+1})\leq 0$.
\end{proof}

A bifunction $f$ is said to be \emph{properly quasi-monotone} on a convex subset $C$ of $X$ if for all finite and non-empty subset $A$ of $C$, and all $x\in \co(A)$
	\[
	\min_{a\in A}f(a,x)\leq0.
	\]

Under quasi-convexity in the second argument, a pseudo-monotone bifunction is also properly quasi-monotone. 
But in general proper quasi-monotonicity does not even imply quasi-monotonicity (see \cite{BP01}).

Another important concept that we will consider in this paper is the upper sign property for bifunctions. A bifunction $f:X\times X\to\R$ is said to have the 
\emph{upper sign property} on a convex subset $C$ of $X$ if, for all $x,y\in C$ 
the following implication holds:
\begin{align}\label{upper}
\bigl(
f(x_t,x)\leq0,~
\forall~ t\in\,]0,1[~
\bigr)
\Rightarrow ~ f(x,y) \geq0.
\end{align}

Recently in \cite{JC-JZ-2018}, the authors showed that under suitable assumptions the upper sign property of $f$ is equivalent to the pseudo-monotonicity of $-f$. 

We finish this section by recalling the Kakutani Fixed Point Theorem (see \cite{GD}), on which our main results are based. 
Assume that $X$ and $Y$ are topological spaces and consider a set-valued map $S:X\tos Y$. The map $S$ is said to be:
\begin{itemize}
 \item \emph{closed} when for any net $(x_i,y_i)\in{i\in I}$ in the graph of $S$ such that $(x_i,y_i)$ converges to $(x_0,y_0)$, it holds $y_0\in S(x_0)$;

 \item \emph{lower semicontinuous (lsc)} when for any $x_0$, and any open set $V$ such that $S(x_0)\cap V\neq\emptyset$, there exists a neighborhood $U$ of $x_0$ such that $S(x)\cap V\neq\emptyset$ for all $x\in U$;
 \item \emph{upper semicontinuous (usc)} when for any $x_0$ and any neighborhood $V$ of $S(x_0)$, there exists a neighborhood $U$ of $x_0$ such that $S(U) \subset V$.
\end{itemize}
Recall also that if $X=Y$ we denote by $\fix(S):=\{x\in X: x\in S(x)\}$ the set of \emph{fixed points} of $S$.

\begin{theorem}[Kakutani]\label{Kakutani}
Assume $X$ is locally convex space and let $C\subset X$ be a non-empty, compact and convex set, and $S: C \tos C$ be a set-valued map. If $S$ is usc and has non-empty, closed and convex images, then $\fix(S)$ is non-empty.
\end{theorem}

\section{The finite intersection property}\label{FIP}

In this section we introduce the notion of finite intersection property and two of its variants, for bifunctions, and we discuss their relation with the generalized monotonicity properties, namely proper quasi-monotonicity and quasi-monotonicity  
in Propositions \ref{proper-fip} and \ref{fip-star-quasi}, 
 respectively.

\begin{definition}
Let $f:X\times X\to\R$ be a bifunction. The bifunction $f$ is said to have: 
\begin{itemize}
\item The \emph{finite intersection property} (fip) on  $C$ a subset of $X$ if, for any  finite and non-empty subset $A$ of $C$, there exists $x\in C$ such that
	\[
	\max_{a\in A}f(a,x)\leq0.
	\]
\item The \emph{star finite intersection property} (fip$^*$) on $C$ a convex subset of $X$  if, for any  finite and non-empty subset $A$ of $C$, there exists $x\in \co(A)$ such that
	\[
	\max_{a\in A}f(a,x)\leq 0.
	\]

\end{itemize}
\end{definition}
Nessah and Tian, in \cite{Nessah-Tian}, introduced a condition called the $\alpha$-locally dominatedness of a bifunction, which corresponds in the case of $\alpha=0$ to a bifunction with the fip by switching the roles of the variables. They discussed the relation of this property with the finite intersection property for families of sets. In fact, for each $x\in C$ we define the set 
\begin{align}\label{FIP-sets}
 F_x:=\{y\in C:~f(x,y)\leq0\}.
\end{align}
Clearly, $f$ has the fip on $C$ 
if and only if, the family of sets $\{F_x\}_{x\in C}$ has the finite intersection property.
Similarly, $f$ has the fip$^*$ on $C$ if and only if for any non-empty and finite subset $A$ of $C$ it holds that
\[
\left(\bigcap_{a\in A}F_a\right)\cap \co(A)\neq\emptyset.
\]

Observe that under fip$^*$ by taking $A_x=\{x\}$ we have $f(x,x)\leq 0$ for every $x\in X$, while fip does not guarantee this in general.

\begin{remark}
	From Proposition \ref{cyclic-q-fip} it is clear that cyclic quasi-monotonicity implies fip, and moreover, if $C$ is a convex set then cyclic quasi-monotonicity implies fip$^*$ and fip$^*$ implies fip. 
	The converses to these implications are in general not true, as shown by the following two simple examples.
\end{remark}

\begin{example}
The bifunction $f(x,y):=xy$,\, for $x,y\in[0,1]$ has the fip. However, $f$ does not have the fip$^*$ on $[0,1]$. Indeed, for $A=\{1\}$ we have $\max_{a\in A}f(a,x)=f(1,1)=1>0$, for all $x\in \co(\{1\})$.
\end{example}

\begin{example}\label{fipstar-but-not-fip-double}
	Let $f:[0,1]\times [0,1]\to \R$ be defined as 
\[	
	f(x,y):=\left\{\begin{array}{ll}
	0,& \text{ if }\ |x-y|\leq 1/2 \\
	1, & \text{ otherwise.} %The 1 could be replaced by anything positive (non-negative) at at least one point in each of the triangles 
	\end{array}\right. 
\]	
	Let us see that $f$ has the fip$^*$. Consider a non-empty and finite set $A\subset [0,1]$. 
	If $\operatorname*{diam} A= \max_{a,b\in A}|a-b|\leq 1/2$, then by taking any point $x\in A$ we obtain $\max_{a\in A}f(a,x)=0$. Otherwise, if $\operatorname*{diam} A= \max_{a,b\in A}|a-b|> 1/2$, then there exist $a_0,a_1\in A$ such that $a_0<1/2+a_1$ and therefore $1/2\in [a_0,a_1]\subset \co A$. So taking $x=1/2$ we have again that $\max_{a\in A}f(a,x)=0.$
	Thus $f$ has the fip$^*$ on $[0,1]$.
	But we observe that $f$ is not cyclic quasi-monotone on $[0,1]$, in fact, not even quasi-monotone, since $f(1,0)=f(0,1)>0$. 
\end{example}

Now we present a simple case of bifunctions with the fip$^{*}$.
\begin{example}\label{Qopt}
Let $h:X\to\R$ be a function and $C$ be a subset of $X$. Consider the bifunction $f:X\times X\to\R$ defined by
\begin{align}\label{Opt-EP}
f(x,y):=h(y)-h(x).
\end{align}
It is clear that $f$ is cyclic monotone, thus cyclic quasi-monotone and due to Proposition \ref{cyclic-q-fip}  it satisfies the finite intersection property. Moreover, if $C$ is convex, then again by Proposition \ref{cyclic-q-fip} we deduce that $f$ satisfies fip$^*$ on $C$.
\end{example}

Following the proof of Proposition 2.1 in \cite{Nessah-Tian} we  will show that a properly quasi-monotone bifunction has the fip$^*$ whenever it is lower semi-continuous on its second argument.

\begin{proposition}\label{proper-fip}
Let  $C$ be a convex and non-empty subset of $X$ (normed space) and $f:X\times X\to\R$  be a bifunction such that for each $x\in C$ the function $f(x,\cdot)$ is lower semi-continuous. If $f$ is properly quasi-monotone on $C$, then   it has the fip$^*$ on $C$. 
\end{proposition}

\begin{proof}
Let us assume by contradiction that $f$ does not have the fip$^*$. So, there exists $\{x_1,...,x_m\}\subset C$ such that for any $x\in K:=\co(\{x_1,...,x_m\})$, we have $$\max_{i=1,...,m}f(x_i,x)>0.$$
By means of the sets $F_{x_i}:=\left\{y\in K: f(x_i,y)\leq 0\right\}$, this can be stated equivalently as 
$\bigcap_{i=1}^m F_{x_i}=\emptyset$. Thus, since the sets $F_{x_i}$ are closed (due to the lower semicontinuity of $f$ in its second argument) then the function $g:K\to\R_+$ defined by 
\[
g(x):=\sum_{i=1}^m d(x,F_{x_i}),
\]
satisfies $g(x)>0$ for all $x\in K$, and is continuous. Further, the function $h:K\to K$ defined as
\[
h(x):=\sum_{i=1}^m\frac{d(x,F_{x_i})}{g(x)}x_i,
\]
is continuous too. By Schauder-Tychonoff Fixed Point Theorem we deduce that there exists $\bar{x}\in K$ such that $h(\bar{x})=\bar{x}$. Consider the set of indices 
\[
J:=\left\{i=1,...,m: d(\bar{x},F_{x_i})>0\right\}
\]
which is non-empty by a simple argument similar to the one used to prove that $g(x)>0$. Then, $\bar{x}\in \co(\{x_i: i\in J\})$ we have that $\min_{i\in J} f(x_i,\bar{x})>0$,
but this contradicts the proper quasi-monotonicity of $f$ applied for the finite set of points $\{x_i\}_{i\in J}$ and its convex combination $\bar{x}$.
\end{proof}

The previous result is also true if we only consider the closeness assumption of the sublevel sets $S_\lambda[f(x,\cdot)]$, with $\lambda=0$ for each $x$.

Analogous to Proposition 1.2 in \cite{BP01}, we will show that fip$^*$ implies quasi-monotonicity under suitable assumptions.

\begin{proposition}\label{fip-star-quasi}
Let $f:X\times X\to\R$ be a bifunction such that $-f$ is semi-strictly quasi-convex in its second argument and $f$ vanishes on the diagonal of $X\times X$. If $f$ has the fip$^*$ on $X$, then it is quasi-monotone on $X$.
\end{proposition}
\begin{proof}
Let $x,y\in X$ such that $f(x,y)>0$. By semi-strict quasi-convexity of $-f(x,\cdot)$ we obtain
\[
f(x,tx+(1-t)y)>0,
\]
for all $t\in]0,1[$. Now, since $f$ has the fip$^*$, we deduce $f(y,x)\leq0$.
\end{proof}

Example \ref{fipstar-but-not-fip-double} %(see also Remark \ref{remfipstar}) 
shows that the semi-strict quasi-convexity of $f$ in its second argument is essential in Proposition \ref{fip-star-quasi}. In fact, the example proposes a bifunction that has the fip$^*$ and vanishes on the diagonal, while it is not quasi-monotone.

The following result states that cyclic quasi-monotonicity implies proper quasi-monotonicity, under quasi-convexity assumption.

\begin{proposition}
Let $C$ be a convex subset of $X$ and $f:X\times X\to\R$ be a bifunction such that $f$ is quasi-convex in its second argument. If $f$ is cyclic quasi-monotone on $C$, then it is properly quasi-monotone on $C$.
\end{proposition}
\begin{proof}
It is a simple and straightforward adaptation of Proposition 4.4 in \cite{HD1999}.
\end{proof}

Note that the quasi-convexity of $f$ in its second argument cannot be dropped from the assumptions. For instance consider the bifunction $f$ defined by \eqref{Opt-EP} which is always cyclically quasi-monotone but it is properly quasi-monotone if and only if the function $h$ is quasi-convex (see part 2 of Proposition 6.2 in \cite{JCYG}).

\section{Equilibrium problems}\label{main1}

We start this section with a result about the existence of solution for  equilibrium problems.
We denote by $\ep(f,C)$ and $\mep(f,C)$ the solution sets of the equilibrium problem \eqref{EP} and the Minty equilibrium problem \eqref{MEP} associated to $f$ and $C$, respectively.

\begin{lemma}\label{meplemma}
Let $C$ be a topological space, and
$f:C\times C\to\R$ be a bifunction and consider the sets $F_x$ as defined in \eqref{FIP-sets}. 
Assume that for each $x\in C$, the set $F_x$ is closed and that there exists $x\in C$ such that the set $F_x$ is compact, and that $f$ has the fip on $C$. Then, $\mep(f,C)$ is non-empty. 	
\end{lemma}
\begin{proof}
It is clear that $\mep(f,C)=\bigcap_{x\in C} F_x$. Since $\{F_x\}_{x\in C}$ has the finite intersection property due to the fact that $f$ has the fip on $C$, and some $F_x$ is compact we deduce that the set $\mep(f,C)$ is non-empty.
\end{proof}
\begin{proposition}\label{EP-fip}
Assume that at least one of the following conditions hold:
\begin{enumerate}
\item The bifunction $-f$ is pseudo-monotone on $C$,
\item $C$ is a convex subset of a vector space, and $f$ has the upper sign property on $C$.
\end{enumerate}
Then $\ep(f,C)\subset \mep(f,C)$.
\end{proposition}
\begin{proof}
The first case was proved in  \cite{ACI} and the second in \cite{JC-JZ-2018}.
\end{proof}

As a consequence we recover the following results.

\begin{corollary}[Proposition 2.1 in \cite{BP05}]
Let $C$ be a non-empty, compact and convex subset of $\R^n$ and $f:\R^n\times\R^n\to\R$ be a bifunction satisfying the following assumptions
\begin{enumerate}
\item $f$ is properly quasi-monotone,
\item for each $y\in C$, the function $f(\cdot,y)$ satisfies the following implication
\[
\inf_{t\in]0,1[}f(tx+(1-t)y,y)\geq0 ~\Rightarrow~f(x,y)\geq0, 
\]
\item for each $x\in C$, the set $F_x$ defined as in \eqref{FIP-sets} is closed and convex,
\item $f$ is quasi-convex with respect to its second argument,
\item $f$ vanishes on the diagonal of $C\times C$,
\item  the following implication holds
\[
(f(x,y)=0~\wedge~f(x,z)<0)~\Rightarrow~f(x,ty+(1-t)z)<0,\mbox{ for all }t\in]0,1[.
\]
\end{enumerate}
Then $\ep(f,C)$ is non-empty.
\end{corollary}
\begin{proof}
By Proposition \ref{proper-fip} and the fact that $C$ is convex, the bifunction $f$ has the fip on $C$. Using Lemma 3 in \cite{castellani2012} and Lemma 2.1 in \cite{ACI} we deduce that $f$ has the upper sign property. Therefore, the result follows from the second part of Proposition \ref{EP-fip}. %By assumptions 2, 4, 5 and 6, the upper sign property is consequence of Lemma 3 in \cite{castellani2012} and Lemma 2.1 in \cite{ACI}. Therefore, the result follows from Proposition \ref{EP-fip}.
\end{proof}

\begin{corollary}[Theorem 2.3 in \cite{SN10}]\label{coro-sosa}
Let $f:C\times C\to \R$ be a bifunction, where  $C$ is a compact convex and non-empty subset of $X$. If the following hold:
\begin{enumerate}
\item for any finite subset $A$ of $C$, and any $x\in \co(A)$, 
\[
\max_{a\in A}f(x,a)\geq0;
\]
\item for each $y$, the set $\{x\in C:~f(x,y)\geq0\}$ is closed;
\end{enumerate}
then $\ep(f,C)$ is non-empty.
\end{corollary}
\begin{proof}
Consider the bifunction $g:C\times C\to\R$ defined as
\begin{align}\label{g-f}
g(x,y):=-f(y,x).
\end{align}
The first assumption is equivalent to the proper quasi-monotonicity of $g$. Moreover, by the second assumption, we have that the set $G_x:=\{y\in C:~g(x,y)\leq0\}$ is closed.
Thus the set $\mep(g,C)$ is non-empty, due to Lemma \ref{meplemma}. Finally, the result follows from the fact that $\ep(f,C)=\mep(g,C)$.
\end{proof}

\begin{corollary}[Theorem 2.7 in \cite{PQK-NHQ-19}]
Assume $X$ is a compact topological space and $f:X\times X\to\R$ a bifunction such that $-f$ is cyclic quasi-monotone with
\[
\{x\in X:~f(x,y)\geq0\}
\]
a closed set for each $y\in X$. Then
$\ep(f,X)\neq\emptyset$.
\end{corollary}
\begin{proof}
First, since $-f$ is cyclic quasi-monotone, then so is $g$, the bifunction defined in \eqref{g-f}. From Proposition \ref{cyclic-q-fip} we deduce that $g$ satisfies the fip. The result follows from Lemma \ref{meplemma} and the fact that $\ep(f,X)=\mep(g,X)$.
\end{proof}
In \cite{PQK-NHQ-19}, the above corollary is referred to as the Weierstrass theorem for bifunctions, because if the bifunction is defined as in \eqref{Opt-EP} we obtain the classical Weierstrass extreme-value theorem.

\begin{corollary}[Lemma 2.3 in \cite{FFB2001}]
Let $C$ be a closed, bounded and convex subset of a reflexive Banach space $X$,  $f:C\times C\to\R$ be a bifunction and $h:X\to\R\cup\{+\infty\}$ be a function such that the set $C\cap \dom(h)$ is  non-empty. If the following assumptions hold
\begin{enumerate}
\item $f$ vanishes on the diagonal of $C\times C$,
\item for every $x\in C\cap\dom(h)$ and every $y\in C\cap\dom(h)$, 
the following implication holds
\[
f(x,y)+h(y)\geq h(x)~\Rightarrow~ f(y,x)+h(x)\geq h(y),
\]
\item for every $z\in C$, $f(z,\cdot)+h(\cdot)$ is lower semi-continuous and semi-strictly quasi-convex in $C$
\item for every $x,y$ in $C$, the function $t\in[0,1]~\mapsto~f(ty+(1-t)x,y)$ is upper semi-continuous at $t=0$,
\item the function $h$ is lower semi-continuous with $\dom(h)$ convex.
\end{enumerate}
Then there exists $x\in C$ such that
\[
f(x,y)+h(y)\geq h(x),\mbox{ for all }y\in C.
\]
\end{corollary}
\begin{proof}
We define the bifunction $g:C_1\times C_1\to\R$, with $C_1:=C\cap\dom(h)$, by
\[
g(x,y):=f(x,y)+h(y)-h(x).
\]
It is clear that $g$ is lower semi-continuous and semi-strictly quasi-convex in its second argument. Also, it is pseudo-monotone on $C_1$, which it turn implies that it is properly quasi-monotone, due to  Proposition 1.1 $(ii)$  in \cite{BP01}. Thus, by Proposition \ref{proper-fip}, $g$ has the fip$^*$ on $C_1$.  

On the other hand, for each $x\in C$ we define the set
\[
G_x:=\{y\in C:~f(x,y)+h(y)\leq h(x)\}.
\]
We note that for each $x\in C\setminus C_1$ the set
$G_x=C$. Thus, it holds
\[
\bigcap_{x\in C}G_x=\bigcap_{x\in C_1}G_x.
\]
The family of sets $\{G_x\}_{x\in C_1}$ has the finite intersection property, due to the fact that $g$ has the fip$^*$ on $C_1$. Since $g(x,\cdot)$ is lower semi-continuous and $C$ is weakly compact, we deduce that $G_x$ is weakly compact, for all $x\in C$. Hence,
\[
\emptyset\neq \bigcap_{x\in C_1}G_x=\mep(f,C_1).
\]
Assumptions 4 and 5 imply the upper sign property of $g$ on $C_1$ and thus, by Proposition \ref{EP-fip}, there exists $x\in \ep(f,C_1$), which means 
 \[
 f(x,y)+h(y)\geq h(x),\mbox{ for all }y\in C_1.
 \]
 Note that for each $y\in C\setminus C_1$, $h(y)=+\infty$. Hence
 \[
f(x,y)+h(y)\geq h(x),\mbox{ for all }y\in C. 
 \]
\end{proof}

\section{Quasi-equilibrium problems}\label{main2}
We state now our main result about the existence of solutions for quasi-equilibrium problems which follows a similar approach to the one proposed in \cite{FFB2001}.

\begin{theorem}\label{P-weak-proper}
Let $f:X\times X\to\R$  be a bifunction, $C$ be a non-empty, convex and compact subset of $X$ 
and $K:C\tos C$ be a set-valued map. If the following assumptions hold:
\begin{enumerate}
\item the map $K$ is closed and lsc, with convex and non-empty values,
\item $f$  has both the upper sign property  and the fip$^*$ on $C$, 
\item the set $M=\{(x,y)\in C\times C:~f(x,y)\leq0\}$ is closed,
\item for each $x\in C$, the set $F_x$ (defined as in \eqref{FIP-sets}) is convex;
\end{enumerate}
then the quasi-equilibrium problem admits at least one solution. 
\end{theorem}
\begin{proof}
We define $g:X\times X\to\R\cup\{+\infty\}$ as
\[
 g(x,y):=i_{K}(x,y)=\left\lbrace\begin{array}{cl}
         0,&y\in K(x)\\
 +\infty,&\mbox{otherwise}
 \end{array}
\right..
\]
Since $K$ is closed, we deduce that $g$ is lower semi-continuous. Moreover,
as $K$ is convex valued, the bifunction $g$ is convex with respect to its second argument.
So, for each $x,w\in C$, we define the set
\[ 
G_x(w):=\{y\in C:~f(w,y)+g(x,y)\leq g(x,w)\}. 
 \]
On the one hand if $w\notin K(x)$, then $G_x(w)=C$. On the other hand, if $w\in K(x)$ we have $G_x(w)=F_w\cap K(x)$. Thus,  $G_x(w)$ is a compact, convex and non-empty subset of $C$. 
Since $f$ has the  fip$^*$ on $C$, for any $w_1,\dots,w_n\in C$, we have 
\[
\bigcap_{i=1}^m G_x(w_i)\neq\emptyset.
\]
Indeed, put $J:=\{i\in\{1,\dots,m\}:~w_i\in K(x)\}$. If $J=\emptyset$, then 
$\bigcap_{i=1}^m G_x(w_i)=C$. Else,
\[
\bigcap_{i=1}^m G_x(w_i)=\bigcap_{i\in J}G_x(w_i).
\] 
Thus, there exists $z\in \co(\{w_i\}_{i\in J})\subset K(x)$ such that
\[
\max_{i\in J}f(w_i,z)\leq0.
\]
Hence  $z\in \bigcap_{i\in J} G_x(w_i)$.

So, for each $x\in C$, the family of sets $\{G_x(w)\}_{w\in C}$ has the finite intersection property. Since each $G_x(w)$ is compact, we have 
 $\bigcap_{w\in C}G_x(w)\neq\emptyset$.
Thus, the set-valued map $S:C\tos C$ defined by
\[
S(x):= \bigcap_{w\in C}G_x(w)
\]
is compact, convex and non-empty valued. We will show now that $S$ is closed.  Indeed, let $(x_i,y_i)_{i\in I}$
be a net in the graph of $S$ such that it converges at $(x,y)$. For all $i\in I$
\[
f(w,y_i)+g(x_i,y_i)\leq g(x_i,w)\mbox{ for all }w\in C. 
\]
Taking $w\in K(x_i)$ we deduce $y_i\in K(x_i)$, which in turn implies $y\in K(x)$. 
As $K$ is lower semi-continuous,
for all $w\in K(x)$ there exists a subnet $(x_{\varphi(j)})_{j\in J}$ of $(x_i)_{i\in I}$ and a net
 $(w_j)_{j\in J}$ converging to $w$ such that
$w_j\in K(x_{\varphi(j)})$ for all $j\in J$. So $f(w_j,y_{\varphi(j)})\leq0$ for all $j\in J$. By the closeness of set $M$, one has $f(w,y)\leq0$. So,
it holds
\[
f(w,y)+g(x,y) \leq g(x,w)\mbox{ for all }w\in C.
\]
Thus, $y\in S(x)$. Additionally, as $S(C)$ is relatively compact, $S$ is upper semi-continuous. 
Thus, $S$ admits at least a fixed point, due to Theorem \ref{Kakutani}, that means there exists
$x_0\in C$ such that 
\[
f(w,x_0)+g(x_0,x_0) \leq g(x_0,w)\mbox{ for all }w\in C. 
\]
Taking $w\in K(x_0)$ in the previous inequality we have $x_0\in K(x_0)$. Therefore, $x_0\in \mqep(f,K)$.
Thus, by Proposition 3.1 in \cite{ACI}, $x_0$ is a solution of the quasi-equilibrium problem.
\end{proof}

As a consequence of Theorem \ref{P-weak-proper} we recover the following result. 
\begin{corollary}[Proposition 4.5 in \cite{AC-2013}]
Let $h:X\to\R$ be a continuous and quasi-convex function, $C$ be a convex, compact and non-empty subset of $X$ and $K:C\tos C$ be a closed and lower semi-continuous set-valued map with convex and non-empty values. Then there exists $x\in \fix(K)$ such that
\[
h(x)\leq h(y),\mbox{ for all }y\in K(x).
\]
\end{corollary}
\begin{proof}
Clearly the bifunction $f$ defined as in Example \ref{Qopt} has the fip$^*$ on $C$ and it is continuous and quasi-convex in its second argument. Moreover,
by the part 2 of Proposition 6.2 in \cite{JCYG}, it has the upper sign property. Thus, Theorem \ref{P-weak-proper} guarantees the existence of a point $x\in \qep(f,K)$, which is equivalent to $x\in K(x)$ and $h(y)\geq h(x)$, for all $y\in K(x)$.
\end{proof}

The problem associated to the previous corollary is well-known in the literature as \emph{quasi-optimisation}.

\begin{remark}
A few remarks are needed
\begin{enumerate}
\item Convexity of the sets $F_x$ is guaranteed by the quasi-convexity in the second argument of the bifunction $f$. Also, closeness of the set $M$ holds if $f$ is lower semi-continuous.
\item It is clear from the proof of Theorem \ref{P-weak-proper} that the assertion remains valid if the fip$^*$ of  $f$ is replaced by the fip on $K(x)$, for all $x\in C$.

\item Theorem \ref{P-weak-proper} is also true if we replace the upper sign property by the pseudo-monotonicity of $-f$.

\item Theorem \ref{P-weak-proper} is strongly related with Theorem 4.5 in \cite{ACI} and Theorem 4.3 in \cite{FFB2001}. However these results are established under generalized monotonicity and quasi-convexity, which are stronger than the finite intersection property.
\end{enumerate}

\end{remark}

\section{Applications}\label{applications}
In this section we consider applications on the study of two particular classes of problems, the variational inequality problem and the generalized Nash equilibrium problem.
\subsection{Variational inequality problem}
Let $X$ be a real Banach space, $X^*$ its topological dual and $\left\langle \cdot,\cdot\right\rangle$ the duality pairing.  
Given a subset $C$ of $X$ and a set-valued map $T:X\tos X^*$, the set $\vip(T,C)$ denotes the solution set of the variational inequality problem
\[
\{x\in C:~\mbox{there exists }x^*\in T(x)\mbox{ such that }\langle x^*,y-x\rangle\geq0\mbox{ for all }y\in C\}.
\]

Associated to a set-valued map $T:X\tos X^*$ with weak$^*$ compact values, we consider the bifunction $f_T:X\times X\to\R$ defined as
\begin{align}\label{T-f}
f_T(x,y):=\sup_{x^*\in T(x)}\langle x^*,y-x\rangle.
\end{align}

An important property of the bifunction $f_T$ is that $\ep(f_T,C)=\vip(T,C)$. Additionally, some of the well-known notions of generalized monotonicity for $T$ can be equivalently defined using the bifunction $f_T$. In fact, $T$ is pseudo-monotone (properly quasi-monotone, quasi-monotone, cyclic quasi-monotone, respectively) if and only if $f_T$ is so.

In this context, since $f_T$ is automatically lower semi-continuous and convex in its second argument, we observe that in this case  pseudo-monotonicity implies proper quasi-monotonicity and the latter implies quasi-monotonicity.

We now recall sign-continuity for set-valued maps. A set-valued map $T:X\tos X^*$ is said to be:
\begin{itemize}
\item \emph{lower sign-continuous} at $x\in{\rm dom}(T)$ if, for any
$v\in X$, the following implication holds:
\[
 \left( \forall t\in]0,1[,~\inf_{x_t^*\in T(x_t)}\langle x_t^*,v\rangle\geq0\right)~\Rightarrow~
 \inf_{x^*\in T(x)}\langle x^*,v\rangle\geq0
\]

\item \emph{upper sign-continuous} at $x\in{\rm dom}(T)$ if, for any
$v\in X$, the following implication holds:
\[
 \left( \forall t\in]0,1[,~\inf_{x_t^*\in T(x_t)}\langle x_t^*,v\rangle\geq0\right)~\Rightarrow~
 \sup_{x^*\in T(x)}\langle x^*,v\rangle\geq0
\]
\end{itemize}
where $x_t=x+tv$.

Obviously, every lower sign-continuous set-valued map is also upper sign-continuous. Note also that the upper sign-continuity of $T$ corresponds to the upper sign property of $f_T$ (see  \cite{ACI}).
In a similar way to Proposition 3 in \cite{JC-JZ-2018}, 
we will show now that if $-T$ is pseudo-monotone, then $T$ is lower sign-continuous, and thus also upper sign-continuous.

\begin{proposition}\label{lower-sign}
Let $T:X\tos X^*$ be a set-valued map. 
If $-T$ is pseudo-monotone, then $T$ is lower sign-continuous on ${\rm int}({\rm dom}(T))$.
\end{proposition}
\begin{proof}
By contradiction, let us assume that $T$ is not lower sign-continuous at $x\in {\rm int}({\rm dom} (T))$, so that there exists $v\in X$ and $x^*\in T(x)$ such that $\langle x^*,v\rangle<0$ and 
\begin{equation}\label{eq:uscprop}
\forall\,t\in]0,1[,\,\inf_{x_t^*\in T(x_t)}\langle x_t^*,v\rangle\geq 0,
\end{equation}
where $x_t=x+tv$.
%Assume that for a given $x\in {\rm int}({\rm dom} (T))$, $x^*\in T(x)$ and $v\in X$ we have $\langle x^*,v\rangle<0$. 
Let $t\in\ ]0,1[\ $ be small enough, so that $x_t=x+tv$ belongs to ${\rm dom} (T)$. Clearly, since $\langle x^*,v\rangle=\langle x^*,t^{-1}(x_t-x)\rangle<0$ then $\langle -x^*,x_t-x\rangle>0$. 
Now, take $x_t^*\in T(x_t)$. The pseudo-monotonicity of $-T$ implies $\langle -x_t^*,x-x_t\rangle<0$, which in turn implies $\langle x_t^*,v\rangle<0$. 
However, this is a contradiction with  \eqref{eq:uscprop}.
\end{proof}

Recently, in \cite{PQK-NHQ-19} the authors proposed a new existence result without the assumption of convexity of the constraint set but under cyclic quasi-monotonicity, which is in particular proper quasi-monotonicity on convex sets, due to Proposition 4.4 in \cite{HD1999}. 
Furthermore, important results about the existence of solution for variational inequality under convexity of the constraint set and generalized monotonicity (more specifically proper quasi-monotonicity and quasi-monotonicity) were established by Aussel and Hadjisavvas in \cite{AH04}.

We now introduce the finite intersection property for set-valued maps. A set-valued map $T:X\tos X^*$ is said to have the \emph{finite intersection property} on a subset $C$ of $X$ if,
for any finite subset  $A$ of $C$ there exists $x\in C$ such that
\[
\langle a^*,x-a\rangle\leq 0,~\forall a\in A,~\forall a^*\in T(a).
\]

\begin{proposition}\label{T-fip-to-fip}
Let $T:X\tos X^*$ be a set-valued map with weak$^*$ compact values, and $f_T$ be the bifunction defined as in \eqref{T-f}. 
$T$ has the finite intersection property if and only if $f_T$ has fip.
\end{proposition}
\begin{proof}
It is a direct consequence from definitions.
\end{proof}

By Proposition \ref{proper-fip}, if $T$ is properly quasi-monotone on a convex set $C$ then $T$ has the finite intersection property on $C$. 

An important class of set-valued maps with the finite intersection property are the adjusted normal operators, which were introduced by Aussel and Hadjisavvas in \cite{AuHad2005}. Indeed, this can be deduced from \cite[Proposition 4.4]{HD1999} and \cite[Proposition 3.3]{AuHad2005}, 

Finally, we have an existence result for variational inequality problems without convexity of the constraint set and under the finite intersection property.
\begin{corollary}\label{coro-vip}
Let $T:X\tos X^*$ be a set-valued map with weak$^*$ compact and non-empty values, and $C$ be a non-empty and compact subset of $X$. If $T$ satisfies the finite intersection property, and one of the following hold
\begin{enumerate}
\item $-T$ is pseudo-monotone,
\item $C$ is convex and $T$ has the upper sign continuity,
\end{enumerate}
  then $\vip(T,C)$ is non-empty.
\end{corollary}
\begin{proof} 
By Proposition \ref{T-fip-to-fip}, the bifunction $f_T$ defined as in \eqref{T-f} has the fip on $C$. Furthermore, it is not difficult to see that $-f_T$ is pseudo-monotone. Since $C$ is compact, Lemma \ref{meplemma} and Proposition \ref{EP-fip} ensures that $\ep(f_T,C)$ is non-empty. The conclusion follows from $\ep(f_T,C)=\vip(T,C)$.
\end{proof}

%We note that if additionally in the previous corollary the set $C$ is convex, it is possible to relax the pseudo-monotonicity  assumption of $-T$ by the upper sign continuity of $T$.
%The following example show that the concept of quasi-monotonicity and the finite intersection property are independents.
%\begin{example}
%Consider the set $C=[0,1]\times[0,1]$ and  $T:C\to C$ defined as
%\[
%T(x,y):=\left\lbrace\begin{array}{cc}
%(0,x),&y=0\\
%(y,0),&x=0\\
%(0,0),& \mbox{otherwise}
%\end{array}\right.
%\]
%which is not quasi-monotone. Indeed, for $(1,0)$ and $(0,1)$ we have
%\[
%\langle T(1,0), (0,1)-(1,0)\rangle =1>0~\mbox{ and }~\langle T(0,1),(0,1)-(1,0)\rangle=-1<0.
%\]
%However, $T$ has the finite intersection property on $C$ and this is a consequence of the following the fact $\langle T(x,y), (0,0)-(x,y)\rangle
%=0$ for all $(x,y)\in C$. Moreover, it is not difficult to show that $T$ has the upper sign continuity on $C$. Therefore, due to Corollary \ref{coro-vip} we deduce $\vip(T,C)\neq\emptyset$. It is important to note that we can not apply the Theorem 2.1 in \cite{AH04} in order to establish the non-emptyness of the set $\vip(T,C)$ . 
%\end{example}

\subsection{Generalized Nash equilibrium problem}
A generalized Nash equilibrium problem (GNEP) consists of
$p$ players. Each player $\nu$ controls the decision variable $x^\nu\in C_\nu$, where $C_\nu$ is a non-empty and closed subset
of a topological space $X_\nu$. We denote by $x=(x^1,\dots,x^p)\in \prod_{\nu=1}^p C_\nu=C$ the vector formed by all these decision
variables and by $x^{-\nu}$ we denote the strategy vector of all the players different from player $\nu$, the \emph{rivals} of $\nu$. 
The set of all such vectors will be denoted by $C^{-\nu}$. We sometimes abuse of the notation by writing $x=(x^\nu,x^{-\nu})$ in order to emphasize player $\nu$-th's variables within $x$. 
%Note that this is still the vector $x=(x^1,\dots,x^\nu,\dots,x^p)$, and the notation $(x^\nu,x^{-\nu})$ does not mean that the block components of $x$ are reordered in such a way that $x^\nu$ becomes the first block.

Each player $\nu$ has an objective function $\theta_\nu:X=\prod_\nu^p X_\nu\to\R$ that depends on all player's strategies. Furthermore,
each players' strategy must belong to a set identified by the set-valued map $K_\nu:C^{-\nu}\tos C_\nu$ in the sense that the strategy space of the player $\nu$ is $K_\nu(x^{-\nu})$ and thus depends on the rivals' strategies $x^{-\nu}$.
Given the strategy vector $x^{-\nu}$ of rivals of player $\nu$, he (or she) chooses a strategy $x^\nu$ that solves the following optimization problem
\begin{align}\label{opt-nu}
\min_{x^\nu} \theta_\nu(x^\nu,x^{-\nu}),\mbox{ subject to }x^\nu\in K_\nu(x^{-\nu}).
\end{align}
%for any given strategy vector $x^{-\nu}$ of the rival players. 
Let  ${\rm Sol}_\nu(x^{-\nu})$ denote the solution set of the problem \eqref{opt-nu}. A {\em generalized Nash equilibrium} is a vector 
$\hat{x}$ such that $\hat{x}^\nu\in {\rm Sol}_\nu(\hat{x}^{-\nu})$, for any $\nu$.

An important situation is the case of joint constraints, which was introduced by Rosen in 1965 (see \cite{Rosen}) and has recently been considered in \cite{Aussel-Dutta,Facchinei2007,SN10}. 
This case is described with a non-empty subset $C$ of $X$ by letting the constraint set-valued maps be defined as
\begin{align}\label{Jointly}
K_\nu(x^{-\nu}):=\{z^\nu\in X_\nu:~(z^\nu,x^{-\nu})\in C\},
\end{align}
for any $\nu$ and $x=(x^\nu,x^{-\nu})\in C$.

Associated to generalized Nash equilibrium problem, we define the following bifunction
$f_0:X\times X\to\R$ as
\begin{align}\label{NI-dual}
f_0(x,y):=\sum_{\nu=1}^p\theta_\nu(y^\nu,y^{-\nu})-\theta_\nu(x^\nu,y^{-\nu})
\end{align}

It is important to note that  $g:X\times X\to\R$ defined as $g(x,y):=-f_0(y,x)$ is the well-known Nikaid\^o-Isoda function introduced in \cite{Nikaido-Isoda}. %Thus, $f_0(x,x)=0$ for all $x\in X$.

The following result says every solution of the Minty equilibrium problem is a solution of the generalized Nash equilibrium problem in the joint case. 
\begin{lemma}\label{Minty-Nash}
Let us assume, for any $\nu$ the subset $K_\nu(x^{-\nu})$ is defined as in \eqref{Jointly}.
Then every solution of $\mep(f_0,C)$ is a generalized Nash equilibrium.
%solution of the GNEP.
\end{lemma}
\begin{proof}
Let $\hat{x}$ be an element of $\mep(f_0,C)$. For each $\nu$ and any $y^\nu\in K_\nu(\hat{x}^{-\nu})$ we have 
\[
\theta_\nu(\hat{x})-\theta_\nu(y^\nu,\hat{x}^{-\nu})=f_0(y,\hat{x})\leq 0,
\]
where $y=(y^\nu,\hat{x}^{-\nu})\in C$, which in turn implies $\theta_\nu(\hat{x})\leq \theta_\nu(y^\nu,\hat{x}^{-\nu})$. The result follows.
\end{proof}

\begin{corollary}\label{Existence-Nash}
Assume that $C$ is  compact and non-empty and for any $\nu$ the subset $K_\nu(x^{-\nu})$ is defined as in \eqref{Jointly}. If $f_0$ defined as in \eqref{NI-dual} has the fip on $C$ and the set $F_x=\{y\in C:~f_0(x,y)\leq0\}$ is closed for all $x\in C$, then there exists a generalized Nash equilibrium.
% GNEP admits a solution.
\end{corollary}
\begin{proof}
It follows from Lemma \ref{meplemma} and Lemma \ref{Minty-Nash}.
\end{proof}

We have that Corollary 4.6 in \cite{PQK-NHQ-19} is a direct consequence of Corollary \ref{Existence-Nash}, thanks to Proposition \ref{cyclic-q-fip} .

The next result establishes sufficient conditions to guarantee the fip$^{*}$ of $f_0$.
\begin{proposition}
Assume that each $X_\nu$ is a topological vector space and the set $C$ is convex.
If each objective function is continuous and convex with respect to the variable of its player,
then the bifunction $f_0$ defined as in \eqref{NI-dual} has fip$^*$ on $C$.
\end{proposition}
\begin{proof}
It is clear $-f_0(\cdot,y)$ is (quasi) convex and $f_0$ vanishes on the diagonal of $X\times X$. By Proposition 1.1 in \cite{BP01}, we deduce $f_0$ is properly quasi-monotone. Since $f_0$ is continuous, the result follows from Proposition \ref{proper-fip}.
\end{proof}

\begin{remark}
An important instance (see \cite{PQK-NHQ-19}) where the bifunction $f$ is cyclically quasi-monotone is when each objective function $\theta_\nu$ has separable variables, that is, it can be written as $\theta_\nu(x^\nu,x^{-\nu}):=f_\nu(x^\nu)+g_\nu(x^{-\nu})$. Indeed, this follows from writing
\[
f(x,y)=\sum_{\nu=1}^p f_\nu(y^\nu)-f_\nu(x^\nu)=\varphi(y)-\varphi(x),
\]
where $\varphi(z)=\sum_{\nu=1}^p f_\nu(z^\nu)$, and Example \ref{Qopt}. 
\end{remark}

\bibliographystyle{abbrv}

%\bibliography{Data-base}

\end{document}